\documentclass[12pt]{amsart}

\usepackage{graphicx}
\usepackage{amsmath}
\usepackage{amssymb}
\usepackage{amsfonts}
\usepackage{verbatim}
\usepackage{scalefnt}
\usepackage{multirow}
\usepackage{enumerate}
\usepackage{mathtools}
\usepackage{float}
\usepackage{enumitem}
\restylefloat{figure}
\usepackage[margin=3cm]{geometry}

\usepackage{fancyhdr}
\usepackage{hyperref}

\newtheorem{theorem}{Theorem}
\newtheorem{lemma}[theorem]{Lemma}
\newtheorem{proposition}[theorem]{Proposition}
\newtheorem{corollary}[theorem]{Corollary}

\newtheorem*{definition*}{Definition}

\numberwithin{equation}{section}
\numberwithin{theorem}{section}

\title[On sumsets of nonbases of maximum size]%
  {On sumsets of nonbases of maximum size}

\author{B\'ela Bajnok}
\email{bbajnok@gettysburg.edu}
\address{Department of Mathematics, Gettysburg College, Gettysburg, PA 17325, USA}

\author{P\'eter P\'al Pach}
\email{ppp@cs.bme.hu}
\address{Department of Computer Science and Information Theory, Budapest University of Technology and Economics, M\H{u}egyetem rkp. 3., H-1111 Budapest, Hungary; MTA-BME Lend\"ulet Arithmetic Combinatorics Research Group,
  ELKH, M\H{u}egyetem rkp. 3., H-1111 Budapest, Hungary.}

\thanks{}

\begin{document}

\begin{abstract}  
Let $G$ be a finite abelian group.  A nonempty subset $A$ in $G$ is called a basis of order $h$ if $hA=G$; when $hA \neq G$, it is called a nonbasis of order $h$.  Our interest is in all possible sizes of $hA$ when $A$ is a nonbasis of order $h$ in $G$ of maximum size; we provide the complete answer when $h=2$ or $h=3$.     	
\end{abstract}

\date{\today}
\maketitle

2020 AMS MSC:  Primary: 11B13; Secondary: 05B10, 11P70, 11B75, 20K01.

Key words: Abelian group, sumset, basis, critical number.

\section{Introduction}

Let $G$ be a finite abelian group of order $n \geq 2$, written in additive notation.  For a positive integer $h$, the {\em Minkowski sum} of nonempty subsets $A_1, \ldots, A_h$ of $G$ is defined as 
$$A_1+  \cdots + A_h = \{ a_1+ \cdots + a_h  \; : \; a_1 \in A_1, \ldots, a_h \in A_h\}.$$  When $A_1= \cdots =A_h=A$, we simply write $hA$, which then is the collection of sums of $h$ not-necessarily-distinct elements of $A$.  

We say that a nonempty subset $A$ of $G$ is {\em $h$-complete} (alternatively, a {\em basis of order $h$}) if $hA=G$; while, if $hA$ is a proper subset of $G$, we say that $A$ is {\em $h$-incomplete}.  The {\em $h$-critical number}  $\chi  (G, h) $ of $G$ is defined as the smallest positive integer $m$ for which all $m$-subsets of $G$ are $h$-complete; that is:
$$\chi  (G, h)  =  \min \{m \; : \; A \subseteq G,  |A| \geq m \Rightarrow h  A=G \}.$$
It is easy to see that for all $G$ and $h$ we have $hG=G$, so $\chi  (G, h)$ is well defined.
The value of  ${\chi}  (G, h)$ is now known for every $G$ and $h$---see \cite{Baj:2018a, Baj:2014a}.  

The following question then arises naturally: What can one say about the size of $hA$ if $A$ is an $h$-incomplete subset of maximum size in $G$?  Namely, we aim to determine the set
$$S(G,h)=\{ |hA| \; : \; A \subset G, \; |A|={\chi}  (G, h)-1, \; hA \neq G\}.$$

In this paper we attain the complete answer to this question for $h=2$ and $h=3$.  For $h=2$, we find that the situation is greatly different for groups of even and odd order.

\begin{theorem} \label{intro thm h=2}
Let $G$ be an abelian group of order $n$.  
\begin{enumerate}
  \item 
When $n$ is even, the maximum size of a 2-incomplete subset of $G$ is $n/2$, and the elements of $S(G,2)$ are of the form $n - n/d$ where $d$ is some even divisor of $n$; in fact all such integers are possible, with the exception that $3n/4$ arises only when the exponent of $G$ is divisible by 4.
  \item  
When $n$ is odd, the maximum size of 2-incomplete subsets of $G$ is $(n-1)/2$; furthermore, when $G$ is of order 3, 5, or is noncyclic and of order 9, then $S(G,2)=\{n-2\}$, and for all other groups of odd order we have $S(G,2)=\{n-2,n-1\}$.
\end{enumerate}
\end{theorem}

For $h=3$ we separate three cases.    
\begin{theorem} \label{intro thm h=3}
Let $G$ be an abelian group of order $n$.  
\begin{enumerate}
  \item 
When $n$ has prime divisors congruent to 2 mod 3, and $p$ is the smallest such prime, the maximum size of a 3-incomplete subset is $(p+1)n/(3p)$, and we have $S(G,3)=\{n-n/p\}$.  
\item When $n$ is divisible by 3 but has no divisors congruent to 2 mod 3, then the maximum size of a 3-incomplete subset is $n/3$, and the elements of $S(G,3)$ are of the form $n - n/d$ or $n-2n/d$ where $d$ is some divisor of $n$ that is divisible by $3$; furthermore, all such integers are possible, with the exceptions of $2n/3$ and $n-2n/d$ when the highest power of 3 that divides $d$ is more than the highest power of 3 that divides the exponent of $G$.
\item In the case when all divisors of $n$ are congruent to $1$ mod 3, then the maximum size of a 3-incomplete subset is $(n-1)/3$, and $S(G,3)=\{n-3,n-1\}$, unless $G$ is an elementary abelian 7-group, in which case $S(G,3)=\{n-3\}$.
\end{enumerate}
\end{theorem}  
We should note that the three cases addressed in Theorem \ref{intro thm h=3} are the same as those used while studying sumfree sets---see \cite{DiaYap:1969a} and \cite{GreRuz:2005a}; in fact, the maximum size of a 3-incomplete set in $G$ agrees with the maximum size of a sumfree set in $G$ when $G$ is cyclic.

Our methods are completely elementary, with Kneser's Theorem as the main tool.  In Section 2 we review some standard terminology and notations and prove some auxiliary results, then in Sections 3 and 4 we prove Theorems \ref{intro thm h=2} and \ref{intro thm h=3}, respectively.

\section{Preliminaries}

Here we present a few generic results that will come useful later.
We will use the following version of Kneser's Theorem.

\begin{theorem}[Kneser's Theorem; \cite{Kne:1953a, Nat:1996a}]
If $A_1, \ldots, A_h$ are nonempty subsets of $G$, and $H$ is the stabilizer subgroup of $A_1+  \cdots + A_h$ in $G$, then
$$|A_1+  \cdots + A_h| \geq |A_1|+ \cdots + |A_h| -(h-1)|H|.$$
\end{theorem}


Our first lemma is a simple application of Kneser's Theorem: 

\begin{lemma} \label{lemma k_2 <=}

Suppose that $G$ is a finite abelian group and that $h$ is a positive integer.  Let $A$ be an $h$-incomplete subset of maximum size in $G$, and let $H$ denote the stabilizer of $hA$ in $G$.  Then both $A$ and $hA$ are unions of full cosets of $H$; furthermore, if $A$ and $hA$ consist of $k_1$ and $k_2$ cosets of $H$, respectively, then
$$k_2 \geq hk_1-h+1.$$

\end{lemma}

\begin{proof}  Consider the sumset $A+H$.  Since we have $$h(A+H)=hA+H=hA \neq G,$$ $A+H$ is $h$-incomplete in $G$.  But $A \subseteq A+H$ and $A$ is an $h$-incomplete subset of maximum size, therefore $A+H=A$, implying that $A$, and thus $hA$, are both unions of cosets of $H$.  By Kneser's Theorem, we have 
$$|hA| \geq h|A|-(h-1)|H|,$$ from which our claim follows.  
\end{proof}

We will also use the following observation:

\begin{lemma} \label{lemma phi inverse}

Suppose that $G$ is a finite abelian group and that $h$ is a positive integer.  Let $H$ be a subgroup of $G$ of index $d$ for some $d \in \mathbb{N}$, and let $\phi$ be the canonical map from $G$ to $G/H$.  Suppose further that $B$ is a subset of $G/H$, and set $A=\phi^{-1}(B)$.  Then $|A|=\frac{n}{d} \cdot |B|$ and $|hA|=\frac{n}{d} \cdot |hB|$.

\end{lemma}

Our next result takes advantage of the fact that the elements of a finite abelian group have a natural ordering.  We review some background and introduce a useful result.

When $G$ is cyclic and of order $n$, we identify it with $\mathbb{Z}_n=\mathbb{Z}/n\mathbb{Z}$.  More generally, $G$ has a unique {\em type} $(n_1,\dots,n_r)$, where $r$ and $n_1, \dots, n_r$ are positive integers  so that $n_1 \geq 2$, $n_i$ is a divisor of $n_{i+1}$ for
 $i=1,\dots,r-1$, and $$G \cong \mathbb{Z}_{n_1} \times \cdots \times \mathbb{Z}_{n_r};$$ here $r$ is the {\em rank} of $G$ and $n_r$ is the {\em exponent} of $G$.

The above factorization of $G$ allows us to arrange the elements in lexicographic order and then consider the `first' $m$ elements in $G$.  Namely, suppose that $m$ is a nonnegative integer less than $n$; we then have unique integers $q_1, \dots, q_r$, so that $0 \leq q_k < n_k$ for each $1 \leq k \leq r$, and $$m=\sum_{k=1}^r q_k n_{k+1} \cdots n_{r}.$$  For simplicity, we assume $q_r \geq 1$, in which case the first $m$ elements in $G$ range from the zero element to $(q_1, \dots, q_{r-1}, q_r-1)$ and thus form the set 
$${\mathcal I}(G,m) = \bigcup_{k=1}^r \{q_1\} \times \cdots \times \{q_{k-1}\} \times \{0,1, \dots, q_k-1\} \times \mathbb{Z}_{n_{k+1}} \times \cdots \times \mathbb{Z}_{n_{r}}.$$

The advantage of considering these initial sets is that their $h$-fold sumsets are also initial sets.  Indeed, assuming for simplicity that $hq_k < n_k$ for each $k$, we find that $h {\mathcal I}(G,m)$ consists of the elements from the zero element to $(hq_1, \dots, hq_{r-1},  hq_r -h )$, and thus
$$h {\mathcal I}(G,m) = {\mathcal I}(G,hm-h+1).$$  

We will also employ a slight modification of ${\mathcal I}(G,m)$ where its last element is replaced by the next one in the lexicographic order.  To avoid degenerate cases, we further assume that $q_r \geq 3$, in which case we have 
$${\mathcal I}^\ast(G,m)= {\mathcal I}(G,m-1) \cup \{(q_1, \dots, q_{r-1}, q_r)\};$$  an easy calculation shows that
$$h{\mathcal I}^\ast(G,m)= {\mathcal I}(G,hm-1) \cup \{(hq_1, \dots, hq_{r-1}, hq_r)\}.$$

We can summarize these calculations, as follows.

\begin{proposition} \label{initial segments}
Suppose that $G$ is of type $(n_1,\dots,n_r)$.  Let $0 \leq m < n$, and let $q_1, \dots, q_r$  be the unique integers with $0 \leq q_k < n_k$ for each $1 \leq k \leq r$ for which $$m=\sum_{k=1}^r q_k n_{k+1} \cdots n_{r}.$$
Let $h$ be a positive integer for which $ h q_k < n_k$ for each $1 \leq k \leq r$.  Then for the $m$-subsets ${\mathcal I}(G,m)$ and ${\mathcal I}^\ast(G,m)$ of $G$ we have the following:
\begin{enumerate}
  \item If $q_r \geq 1$, then $|h {\mathcal I}(G,m)| = hm-h+1$. 
  \item If $q_r \geq 3$, then $|h{\mathcal I}^\ast(G,m)|=hm.$
\end{enumerate}
\end{proposition}

\section{Two-fold sumsets}

In this section we prove Theorem~\ref{intro thm h=2}.  We separate two cases depending on the parity of the order of the group: the even case is considered in Theorem~\ref{h=2,even} and the odd case is established in Theorem~\ref{h=2,odd}.

We start by determining the critical number $\chi (G,2)$.

\begin{proposition}  \label{2-critical prop}

For any abelian group $G$ of order $n$ we have $$\chi (G,2)= \left \lfloor n/2 \right \rfloor +1.$$

\end{proposition} 

\begin{proof}  Suppose that $A$ is a subset of $G$ of size $|A| > n/2$.  Since $A$ and $g-A$ cannot be disjoint then for any $g \in G$, we have $2A=G$.  

To complete the proof, we need to identify a subset of $G$ of size $\left \lfloor n/2 \right \rfloor$ that is 2-incomplete.  When $n$ is even, any subgroup of index 2 (or a coset of such subgroup) will do.  

Suppose now that $n$ is odd, in which case $G$ has type $(n_1, \dots, n_r)$ for some $r, n_1, \dots, n_r \in \mathbb{N}$ and $n_k$ odd for all $k$.  We then have
$$\frac{n-1}{2}=\sum_{k=1}^r \frac{n_k-1}{2} \cdot n_{k+1} \cdots n_r.$$  Therefore, according to Proposition \ref{initial segments}, the initial segment  
${\mathcal I}(G,(n-1)/2)$ has a 2-fold sumset of size $n-2$ and is thus $2$-incomplete.     
\end{proof}

We now turn to finding $$S(G,2)=\{|2A| \; : \; A \subset G, \; |A|=\left \lfloor n/2 \right \rfloor, \; 2A \neq G\}.$$
We start with a result that may be of independent interest.

\begin{theorem}  \label{two unequal components}
Let $G$ be a group of even order whose exponent is not divisible by $4$, and suppose that $A$ is a subset of $G$ of size $|A|=n/2$.   Then $G$ has a subgroup $H$ of order $n/2$ for which $$|A \cap H| \neq |A \cap (G \setminus H)|.$$

\end{theorem}

\begin{proof}  We proceed indirectly, and assume that each subgroup of order $n/2$ in $G$ contains exactly half of the elements of $A$.  We may assume that $G=G_1 \times G_2$, where $G_1$ has odd order, and $G_2=\mathbb{Z}_{n_1} \times \cdots \times \mathbb{Z}_{n_r}$ with all $n_i$ even; by assumption, we also know that they are not divisible by $4$.

We say that a subset $C$ of $G$ of the form $C=G_1 \times B_1 \times \cdots \times B_r$ is a {\em projection} of $G$, if for each $i$, either $B_i=\mathbb{Z}_{n_i}$ or $B_i$ is a coset of the subgroup of index 2 in $\mathbb{Z}_{n_i}$.  Note that each projection of $G$ has size $n/2^k$ for some $0 \leq k \leq r$.  We prove the following:

\bigskip

{\bf Claim:}  If $C$ is a projection of $G$ of size $n/2^k$, then $A \cap C$ has size $n/2^{k+1}$.

Since this is clearly impossible for $k=r$, we arrive at a contradiction.

\medskip

{\em Proof of Claim:} We use induction on $k$.  The claim trivially holds for $k=0$, and it also holds for $k=1$, since any projection of $G$ of size $n/2$ is either a subgroup of index 2 or a coset of that subgroup  and, by our indirect assumption, both contain exactly $n/4$ elements of $A$.

Assume now that our claim holds for $k-1$ for some $k \leq r$.  To prove our claim for $k$, by symmetry it clearly suffices to consider projections in 
$${\mathcal C}=\{G_1 \times B_1 \times \cdots \times B_r \; : \; |B_i|=n_i/2 \; \mbox{for} \; 1 \leq i \leq k \; \mbox{and} \; |B_i|=n_i \; \mbox{for} \; k+1 \leq i \leq r\}.$$

Recall that the elements of $\mathbb{Z}_2^k$ may be arranged in Gray-code order; that is, we have a sequence $$e_0, e_1, \dots, e_{2^k-1}, e_{2^k}$$ where $e_0=e_{2^k}$ is the zero-element of $\mathbb{Z}_2^k$, and  $e_j$ and $e_{j+1}$ differ in exactly one position for every $j=0,1,\dots, 2^k-1$.  We can then arrange the elements of ${\mathcal C}$ in a corresponding sequence  $$C_0, C_1, \dots, C_{2^k-1}, C_{2^k}$$ where $C_j=G_1 \times B_1 \times \cdots \times B_r$ has $B_i \leq \mathbb{Z}_{n_i}$ for some $1 \leq i \leq k$ if, and only if, the $i$-th component of $e_j$ equals $0$ (and $(\mathbb{Z}_{n_i} \setminus B_i) \leq \mathbb{Z}_{n_i}$ otherwise).  

Observe that, for every $j=0,1,\dots, 2^k-1$, the union of $C_j$ and $C_{j+1}$ is a projection of $G$ of size $n/2^{k-1}$; therefore, by our inductive hypothesis, it must contain exactly $n/2^{k}$ elements of $A$.  Thus, if $C_0$ contains $t$ elements of $A$, then $C_j$ will contain $t$ elements of $A$ if $j$ is even, and $n/2^k-t$ elements of $A$ when $j$ is odd.  We need to show that $t=n/2^{k+1}$.

It is not hard to see (by a simple parity argument) that $$H=C_0 \cup C_2 \cup C_4 \cup  \cdots \cup C_{2^k-2}$$ is a subgroup of index $2$ in $G$, so by our assumption, it contains $n/4$ elements of $A$.  Therefore, $t \cdot 2^k/2  =n/4$, which proves our claim.  
\end{proof}

We note that the claim of Theorem \ref{two unequal components} may be false in groups with exponent divisible by $4$.  For example, in $\mathbb{Z}_2 \times \mathbb{Z}_4$, the set $\mathbb{Z}_2 \times \{0,1\}$ intersects all three subgroups in two elements.

We are now ready to determine $S(G,2)$.  We start with the case when $n$ is even.

\begin{theorem}\label{h=2,even}

If the exponent of $G$ is divisible by $4$, then $$S(G,2)=\left \{ n-n/d \; : \; d|n, \; 2|d \right \};$$ if the exponent of $G$ is even but not divisible by $4$, then 
$$S(G,2)=\left \{ n-n/d \; : \; d|n, \; 2|d, \; d \neq 4 \right \}.$$

\end{theorem}

{\em Proof:}  Using the notations of Lemma \ref{lemma k_2 <=}, we have $|A|=n/2=k_1n/d$ where $d$ is the index of the stabilizer subgroup of $2A$.  This implies that $d$ is even and $k_1=d/2$; using Lemma \ref{lemma k_2 <=} again yields $k_2 \geq d-1$ and thus $|2A| = k_2n/d$ equals $n$ or $ n -n/d$.  Therefore, we have $$S(G,2) \subseteq \left \{ n-n/d \; : \; d|n, \; 2|d \right \}.$$

When the exponent of $G$ is congruent to $2$ mod $4$, then we can rule out $d=4$, as follows.  By Theorem \ref{two unequal components}, $G$ has a subgroup $H$ of index $2$ for which $H \cap A$ and $ (G \setminus H) \cap A$ have different sizes; let $A=A_1 \cup A_2$ where $A_1$ and $A_2$ are subsets of different cosets of $H$.  Without loss of generality, we assume that $|A_1| > n/4$, and thus $2A_1=H$.  If $A_2$ were to be empty, then $A$ is a full coset of $H$, and thus $|2A| =n/2 \neq 3n/4$.  Otherwise, $|A_1+A_2| \geq |A_1|>n/4$, which implies that $|2A| \geq |2A_1|+|A_1+A_2|>3n/4$. 

What remains is the proof that all remaining values arise as sumset sizes.  This is clearly true when $d=2$, or when $d=4$ and the exponent of $G$ is divisible by 4.  Suppose now that $d$ is an even divisor of $n$ and $d>4$.  According to Lemma \ref{lemma phi inverse}, it suffices to prove that every group $K$ of order $d$ contains some subset $B$ of size $d/2$ for which $|2B|=d-1$.  Let $H$ be any subgroup of index 2 in $K$, and set $B=(H \setminus \{h\}) \cup \{g\}$, where $h$ and $g$ are arbitrary elements of $H$ and $K \setminus H$, respectively.  Since $|H \setminus \{h\}|=d/2-1 > d/4$, we get $2(H \setminus \{h\})=H$ and thus $2A=G \setminus \{h+g\}$.  Therefore, $|2B|=d-1$, and our proof is complete.  {\hfill $\Box$}

Let us now turn to the case when $n$ is odd.

\begin{theorem}\label{h=2,odd}
If $G \cong \mathbb{Z}_3$, $\mathbb{Z}_5,$ or $\mathbb{Z}_3^2$, then $S(G,2)=\{n-2\}$.  For all other $G$ of odd order we have $S(G,2)=\{n-2,n-1\}$.

\end{theorem}

\begin{proof}  Let $A$ be a subset of $G$ of size $(n-1)/2$.  By Lemma \ref{lemma k_2 <=}, $A$ is the union of some $k_1$ cosets of the stabilizer $H$ of $2A$; if $H$ has index $d$ in $G$, then we thus have $(n-1)/2=|A|=k_1n/d$.  But this implies that $d=n$ and $k_1=(n-1)/2$, so using Lemma \ref{lemma k_2 <=} again, we get that $2A$ has size $k_2 \geq n-2$.  Therefore, $S(G,2) \subseteq \{n-2,n-1\}$.

In the proof of Proposition \ref{2-critical prop} we already established that $n-2 \in S(G,2)$ by pointing out that the set ${\mathcal I}(G,(n-1)/2)$, consisting of the initial $(n-1)/2$ elements in $G$, has a 2-fold sumset of size $n-2$.  Similarly, Proposition \ref{initial segments} yields that, when $(n_r-1)/2 \geq 3$, then ${\mathcal I}^\ast(G,(n-1)/2)|$ is of size $(n-1)/2$ and has $|2{\mathcal I}^\ast(G,m)|=n-1.$  

This leaves us with the elementary abelian 3-groups and 5-groups.  When $r \geq 3$, for $\mathbb{Z}_3^r$ we may take the first $(n-1)/2$ elements, except that we replace $(1,1,\dots,1,0,2,2)$ by $(1,1,\dots,1,2,0,0)$; one can easily determine that this way $2A=\mathbb{Z}_3^r \setminus \{(2,2,\dots,2)\}$.  Similarly, when $r \geq 2$, for $\mathbb{Z}_5^r$ we may take the first $(n-1)/2$ elements, except that we replace $(2,2,\dots,2,1,4)$ by $(2,2,\dots,2,3,0)$; this way $2A=\mathbb{Z}_5^r \setminus \{(4,4,\dots,4)\}$.  It can also be readily verified that for $\mathbb{Z}_3$, $\mathbb{Z}_5,$ or $\mathbb{Z}_3^2$, we do not have $n-1 \in S(G,2)$.
\end{proof}

\section{Three-fold sumsets}

In this section we prove Theorem~\ref{intro thm h=3}.  We consider three cases: Theorem~\ref{S(G,3) when 2 mod 3 exists} covers the cases when the order $n$ of the group has some prime divisors that are congruent to 2 mod 3, Theorem~\ref{thm S(G,3) when 3|n but no 2 mod 3} deals with the cases when $n$ is divisible by 3 but has no divisors that are congruent to 2 mod 3, and Theorem~\ref{thm S(G,3) when all divisors are 1 mod 3} and Corollary~\ref{cor-z7n} establish the cases when all divisors of $n$ are congruent to 1 mod 3.

Our first task is to find the 3-critical number of each finite abelian group.

\begin{proposition}  \label{3-critical prop}

Suppose that $G$ is an abelian group of order $n$.  Then:
$$\chi (G,3) =\left\{
\begin{array}{cl}
\left(1+\frac{1}{p}\right) \frac{n}{3} +1 & \mbox{if $n$ has prime divisors congruent to $2$ mod $3$,} \\ & \mbox{and $p$ is the smallest such divisor,}\\ \\
\left\lfloor \frac{n}{3} \right\rfloor +1 & \mbox{otherwise.}
\end{array}\right.$$

\end{proposition}

\begin{proof}  It is easy to see that the expressions above provide lower bounds for $\chi (G,3)$.  Indeed, if $H$ is a subgroup of $G$ of prime index $p$ then $G/H$ is cyclic; by Lemma \ref{lemma phi inverse}, taking an arithmetic progression of size $\left \lfloor (p+1)/3 \right \rfloor$ in $G/H$ yields a set of size $\left \lfloor (p+1)/3 \right \rfloor \cdot n/p$ in $G$ whose 3-fold sumset has size
$$\left( 3 \cdot \left \lfloor \frac{p+1}{3} \right \rfloor -2 \right) \cdot \frac{n}{p},$$ which is less than $n$.  This establishes the cases when $n$ has prime divisors congruent to $2$ mod $3$, and $p$ is the smallest such divisor, or when $n$ is divisible by 3 (take $p=3$).  

For the case when all divisors of $n$ are congruent to $1$ mod $3$, let $(n_1,n_2, \dots, n_r)$ be the type of $G$, and note that 
$$\frac{n-1}{3}=\sum_{k=1}^r \frac{n_k-1}{3} \cdot n_{k+1} \cdots n_r.$$  Therefore, according to Proposition \ref{initial segments}, the initial segment  
${\mathcal I}(G,(n-1)/3)$ in $G$ has a 3-fold sumset of size $n-3$ and is thus $3$-incomplete.

We now show that the expressions above are upper bounds.  Suppose that $A \subseteq G$ is a 3-incomplete subset of maximum size in $G$.  Using the notations of Lemma \ref{lemma k_2 <=}, we have $|A|=k_1n/d$ and $|3A|=k_2n/d$ where $d$ is the index of the stabilizer subgroup of $3A$.  According to Lemma 
\ref{lemma k_2 <=}, $k_2 \geq 3k_1-2$, and since $3A \neq G$, we have $k_2 \leq d-1$, so $k_1 \leq  (d+1)/3.$  

We consider first the case when $n$ has prime divisors congruent to $2$ mod $3$, and $p$ is the smallest such divisor.  In this case we find that
$$|A|= k_1n/d \leq (d+1)/3 \cdot n/d  \leq (1+1/p) \cdot n/3,$$ as claimed.
However, if $n$ has no divisors congruent to $2$ mod $3$, then $k_1 \leq \lfloor d/3 \rfloor$, so
$$|A|= k_1n/d \leq \lfloor d/3 \rfloor \cdot n/d \leq \lfloor n/3 \rfloor,$$ which completes the proof. 
\end{proof}

In the rest of this section we determine $S(G,3)$ for each group $G$.  We start with the case when $|G|=n$ has prime divisors congruent to $2$ mod $3$ and $p$ is the smallest such divisor.  

\begin{theorem}  \label{S(G,3) when 2 mod 3 exists}

Suppose that $n$ has prime divisors congruent to $2$ mod $3$, and $p$ is the smallest such divisor.  Then $S(G,3)=\{n-n/p\}$.

\end{theorem}

\begin{proof}  Suppose that $A$ is a 3-incomplete subset of maximum size in $G$.  Using the notations of Lemma \ref{lemma k_2 <=}, we have $|A|=(p+1)/3 \cdot n/p=k_1n/d$ where $d$ is the index of the stabilizer subgroup of $3A$.  This implies that $d$ is divisible by $p$.  Furthermore, $k_1=(p+1)/p \cdot d/3$; using Lemma \ref{lemma k_2 <=} again yields 
$$k_2 \geq 3k_1 - 2 = d+(d/p-2) \geq d-1,$$ with equality only if $d=p$.  Therefore, $|3A|$ equals $n$ or $n-n/p$, proving that $S(G,3) \subseteq \{n-n/p\}$.

As $S(G,3)\ne \emptyset$ (according to its definition), it is obtained that $S(G,3)=\{n-n/p\}$.


\end{proof}

As a special case of Theorem~\ref{S(G,3) when 2 mod 3 exists}, we see that when the order $n$ of $G$ is odd but divisible by 5, then a 3-incomplete subset of maximum size $0.4n$ in $G$ consists of two cosets of a subgroup of index 5.  It is worth mentioning that, according to a result of Lev in \cite{Lev:2016a}, if $G$ is an elementary abelian 5-group, then any 3-incomplete subset of size at least $0.3n$ is contained in a union of two cosets of a subgroup of index $5$. 

Next, we address the case when the order $n$ of $G$ is divisible by 3 but has no divisors that are congruent to 2 mod 3.

\begin{theorem}  \label{thm S(G,3) when 3|n but no 2 mod 3}
Suppose that $n$ is divisible by $3$ but has no prime divisors congruent to $2$ mod $3$.   We then have
$$S(G,3) = \left \{ n-n/d \; : \; d|n, \; 3|d, \; d \neq 3 \right \} \cup \left \{ n-2n/d  \; : \; d|n, \; 1 \leq \nu_3(d) \leq \nu_3(\kappa) \right \},$$
where $\kappa$ is the exponent of $G$, and $\nu_3(t)$ is the highest power of 3 that divides the integer $t$.

\end{theorem}

\begin{proof} By Proposition \ref{3-critical prop}, the maximum size of a 3-incomplete subset of $G$ in this case is $n/3$.   We provide the proof through several claims.

\bigskip

\noindent {\bf Claim 1:} $S(G,3) \subseteq \left \{ n-cn/d \; : \; d|n, \; 3|d, \; c = 1,2 \right \}.$

\medskip

\noindent {\em Proof of Claim 1:}  Using the notations of Lemma \ref{lemma k_2 <=}, we have $|A|=n/3=k_1n/d$ where $d$ is the index of the stabilizer subgroup of $3A$.  This implies that $d$ is divisible by $3$ and $k_1=d/3$; using Lemma \ref{lemma k_2 <=} again yields $k_2 \geq d-2$ and thus $|3A| = k_2n/d$ equals $n$, $ n -n/d$, or $n-2n/d$, proving our claim.

\bigskip

\noindent {\bf Claim 2:}  If $d$ is a divisor of $n$ that is divisible by 3 and $d \neq 3$, then $n-n/d \in S(G,3)$.

\medskip

\noindent {\em Proof of Claim 2:}  By Lemma \ref{lemma phi inverse}, it suffices to prove that all groups $K$ of order $d$ with $3|d$ and $d>3$ contain some subset $A$ of size $d/3$ for which $|3A|=d-1$.  Let $H$ be any subgroup of index 3 in $K$, and set $A=(H \setminus \{h\}) \cup \{g\}$, where $h$ and $g$ are arbitrary elements of $H$ and $K \setminus H$, respectively.  Note that $d \neq 6$ since $d$ has no divisors congruent to $2$ mod $3$, and thus we have $d \geq 9$.  Therefore, $|H \setminus \{h\}|=d/3-1 > d/6$, so $2(H \setminus \{h\})=H$ and $3(H \setminus \{h\})=H$.  But then 
$$3A=3(H \setminus \{h\}) \cup ((2(H \setminus \{h\}) +g) \cup ((H \setminus \{h\})+2g)=G \setminus \{h+2g\}.$$  Therefore, $|3A|=d-1$, as claimed. 

\bigskip

\noindent {\bf Claim 3:}  We have $2n/3 \not \in S(G,3)$.

\medskip

\noindent {\em Proof of Claim 3:}  As before, we see that $A$ is the union of $k_1=d/3$ cosets of $H$ and $3A$ is the union of $k_2 \geq d-2$ cosets of $H$, where $d$ is the index of the stabilizer subgroup $H$ of $3A$.  But $2n/3=k_2n/d \geq (d-2)n/d$ yields $d \leq 6$, and since $d$ is odd and is divisible by 3, this can only happen if $d=3$.  Therefore, $k_1=1$ and thus $k_2=1$ as well, which gives $|3A|=n/3$.

\bigskip

\noindent {\bf Claim 4:}  If $d$ is a divisor of $n$ for which $\nu_3(d) > \nu_3(\kappa)$, then $n-2n/d \not \in S(G,3)$.

\medskip

\noindent {\em Proof of Claim 4:}  
For the sake of a contradiction, let us assume that  $A$ is a subset of $G$ of size $n/3$ and  $|3A|=n-2n/d$.

Suppose that $H$ is the stabilizer of $3A$ and that $H$ has index $\delta$ in $G$; we will first show that $\delta=d$. According to Lemma~\ref{lemma k_2 <=}, the set $A$ is the union of $k_1=\delta/3$ cosets of $H$, and $3A$ is the union of $\delta-2\delta/d=k_2\geq 3k_1-2$ cosets of $H$. Hence, $d \geq \delta$ and $d$ divides $2\delta$, thus $d$ is either $\delta$ or $2\delta$; since $n$ is odd, we obtain $d=\delta$.

Let $\phi$ be the canonical map from $G$ to $G/H$.  With the notations $G'=G/H$ and 
$A'=\phi(A)$, we then have $|G'|=d$, $|A'|=d/3$, and $|3A'|=d-2$. 

We let $\{x,y\}= G'\setminus (3A')$, and note that $x-A'\subseteq G'\setminus 2A'$ and $y-A'\subseteq G'\setminus 2A'$. 
Since the stabilizer of $3A'$ in $G'$ is trivial, so is the stabilizer of $2A'$, and thus by Kneser's Theorem we have 
$$|G'\setminus 2A'|\leq |G'|-2|A'|+1=d/3+1.$$  This means that $x-A'$ and $y-A'$ have at least $d/3-1$ elements in common.  

Now let $\ell=x-y$, $K=\langle \ell \rangle$, and $|K|=k$.  Since $$|A' \cap (A'+\ell)| =|(x-A') \cap (y-A')| \geq |A'|-1,$$  
$A'$ is the union of arithmetic progressions, each of difference $\ell$, and at most one of them has size less than $k$.  
According to our assumption, $\nu_3(d) > \nu_3(\kappa)$,  so $d/3$ is divisible by $k$, which then means that 
$A'$ is the union of full cosets of $K$.  Therefore, $3A'$ is the union of full cosets of $K$ as well, and thus $d-2$ is divisible by $k$.  But then $k \leq 2$, and thus $k=1$ since $k$ is odd, which is a contradiction if $x \neq y$.    

\bigskip

\noindent {\bf Claim 5:}  If $d$ is a divisor of $n$ for which $1 \leq \nu_3(d)\leq  \nu_3(\kappa)$, then $n-2n/d \in S(G,3)$.

\medskip

\noindent {\em Proof of Claim 5:}  Suppose that $G$ is of type $(n_1,\dots,n_r)$; we can then find positive integers $d_1, \dots, d_r$ so that $d_i | n_i$ for each $i=1, \dots, r$; $d_1 \cdots d_r=d$; and $d_1, \dots, d_{r-1}$ are all congruent to $1$ mod $3$.  We then have
$$\frac{d}{3} = \sum_{k=1}^{r-1} \frac{d_k-1}{3} d_{k+1} \cdots d_r + \frac{d_r}{3}.$$

Let $H$ be a subgroup of $G$ so that $K=G/H$ is of type $(d_1,\dots,d_r)$.  According to Proposition~\ref{initial segments}, the initial segment ${\mathcal I}(K,d/3)$ of size $d/3$ has 3-fold sumset of size $d-2$.  
By Lemma \ref{lemma phi inverse}, $G$ then contains a subset of size $n/3$ whose $3$-fold subset has size $n-2n/d$.  

This completes the proof of Theorem \ref{thm S(G,3) when 3|n but no 2 mod 3}.  
\end{proof}

For our final case, we consider groups whose order $n$ only has divisors that are congruent to 1 mod 3.  Our previous techniques work well for all groups in this category, other than elementary abelian 7-groups, so we consider those separately.

\begin{theorem}  \label{thm S(G,3) when all divisors are 1 mod 3}
If all divisors of the order $n$ of $G$ are congruent to $1$ mod $3$, but $G$ is not isomorphic to an elementary abelian 7-group, then $S(G,3)=\{n-3,n-1\}$.

\end{theorem}

\begin{proof}
By Proposition \ref{3-critical prop}, the maximum size of a 3-incomplete subset of $G$ in this case is $(n-1)/3$.   We provide the proof through the following three claims.

\bigskip

\noindent {\bf Claim 1:} $S(G,3) \subseteq \left \{ n-3, n-2, n-1 \right \}.$

\medskip

\noindent {\em Proof of Claim 1:}  Using the notations of Lemma \ref{lemma k_2 <=}, we have $|A|=(n-1)/3=k_1n/d$ where $d$ is the index of the stabilizer subgroup of $3A$.  This implies that $d$ is divisible by $n$ and thus $d=n$ and $k_1=(n-1)/3$; using Lemma \ref{lemma k_2 <=} again yields $k_2 \geq n-3$, as claimed.

\bigskip

\noindent {\bf Claim 2:}  We have $\{n-3 , n-1 \} \subseteq S(G,3)$.

\medskip

\noindent {\em Proof of Claim 2:}  Suppose that $G$ is of type $(n_1,n_2, \dots, n_r)$.  Since $n_1, \dots, n_r$ are all congruent to 1 mod 3, we have
$$\frac{n-1}{3}=\sum_{k=1}^r \frac{n_k-1}{3} \cdot n_{k+1} \cdots n_r.$$  Therefore, Proposition \ref{initial segments} yields that $n-3 \in S(G,3)$ and, since $n_r \geq 10$, $n-1 \in S(G,3)$ as well.

\bigskip

\noindent {\bf Claim 3:}  We have $n-2 \not \in S(G,3)$.
 
 \medskip

\noindent {\em Proof of Claim 3:}  
Suppose that $A$ is a subset of $G$ of size $(n-1)/3$,  and assume indirectly that $3A=G\setminus\{x,y\}$ with some $x, y \in G$, $x\ne y$.  

According to Lemma~\ref{lemma k_2 <=}, the size of the stabilizer of $3A$ divides both $|A|=(n-1)/3$ and $|3A|=n-2$, therefore it is trivial. 
Then so is the stabilizer of $2A$, so by Kneser's Theorem, 
$$|G\setminus 2A|\leq |G|-2|A|+1=|A|+2.$$  Since $x-A$ and $y-A$ are both of size $(n-1)/3$ and are subsets of $G\setminus 2A$, this then means that they must have at least $|A|-2$ elements in common.

Now let $\ell=x-y$, $K=\langle \ell \rangle$, and $|K|=k$.  Since $$|A \cap (A+\ell)| =|(x-A) \cap (y-A)| \geq |A|-2,$$  
$A$ is the union of arithmetic progressions, each of difference $\ell$, and at most two of them have size less than $k$.  Furthermore, note that $(n-1)/3  \equiv (k-1)/3$ mod $k$.   
Therefore, we have three possibilities:
 \begin{enumerate} 
\item $A$ is the union of some complete cosets of $K$ and an arithmetic progression of size $(k-1)/3$;

\item $A$ is the union of some complete cosets of $K$ and two arithmetic progressions that are in different cosets of $K$, and the sizes of these two arithmetic progressions add to $(k-1)/3$ or $k+(k-1)/3$; or

\item $A$ is the union of some complete cosets of $K$ and two (disjoint) arithmetic progressions that are in the same coset of $K$, and the sizes of these two arithmetic progressions add to $(k-1)/3$.
\end{enumerate}

We can quickly rule out the first case as that would lead to $|3A| \equiv k-3$ mod $k$, contradicting $|3A|=n-2$.

For the second case, suppose that the two arithmetic progressions that are not full cosets of $K$ are $B_1$ and $B_2$, with $|B_1|=r_1$ and $|B_2|=r_2$.
Observe that if $B_1$ and $B_2$ are within distinct cosets of $K$, then so are $3B_1, 2B_1+B_2, B_1+2B_2,$ and $3B_2$.  When $r_1+r_2=(k-1)/3$, then each of these four sumsets have size less than $k$, so we have 
$$n-2=|3A| \equiv |3B_1|+|2B_1+B_2|+|B_1+2B_2|+|3B_2|=6(r_1+r_2)-8 \equiv -10$$ mod $k$.    
This implies that $8$ is divisible by $k$, and since $k>1$, this means that $k$ is even, which is not possible since $k$ is odd.  If $r_1+r_2=k+(k-1)/3$, 
then at least three of the sets $3B_1, 2B_1+B_2, B_1+2B_2,$ and $3B_2$ have size $k$. Indeed, by symmetry we may assume that we have $r_1 \geq r_2$, in which case
$$3r_1-2 \geq 2r_1+r_2-2 \geq r_1+2r_2-2 =k+(k-1)/3+r_2-2 \geq k.$$  Therefore, if $3r_2 -2 < k$, then $n-2=|3A| \equiv 3r_2 - 2$ mod $k$, but that is a contradiction, since $r_2$, and therefore $3r_2$, is not divisible by $k$, and if $3r_2 -2 \geq k$, then $n-2=|3A| \equiv 0$ mod $k$, contradicting that $k>1$ is odd.

Let us now turn to case (3), where $A$ contains arithmetic progressions $B_1$ and $B_2$ that are in the same coset of $K$ and have a combined size of $(k-1)/3$.  It suffices to show that it is not possible that $3(B_1\cup B_2)$ has size $k-2$, and this can be accomplished by proving that if $I_1$ and $I_2$ are disjoint intervals in the cyclic group $\mathbb{Z}_k$ with $|I_1|+|I_2|=(k-1)/3$, then $|3(I_1 \cup I_2)| \neq k-2$.  

Without loss of generality, we may assume that 
$$I_1=\{0,1,\dots,r_1-1\}$$ and $$I_2=\{s,s+1,\dots,s+r_2-1\}$$ for some positive integers $r_1, r_2$, and $s$ with $r_1+r_2=(k-1)/3$, $r_1 \geq r_2$, and $r_1+1 \leq s \leq k-r_2-1$.  Also, we may further assume that $ s\leq (k-1)/3 +r_1$, which holds when among the two gaps between $I_1$ and $I_2$, the size of $\{r_1,r_1+1\dots,s-1\}$ is at most as much as the size of $\{s+r_2, s+r_2+1, \dots, k-1\}$.

The set $3(I_1+I_2)$ is the union of four intervals:
$$3I_1=\{0,1,\dots,3r_1-3\},$$
$$2I_1+I_2=\{s,s+1,\dots,s+2r_1+r_2-3\},$$
$$I_1+2I_2=\{2s,2s+1,\dots,2s+r_1+2r_2-3\},$$
$$3I_2=\{3s,3s+1,\dots,3s+3r_2-3\}.$$

Now if $r_1+1 \leq s\leq (k-1)/3+ r_2-2$, then there is no gap between these intervals, thus they cover (as integer intervals) $[0,3s+3r_2-3]$. Since 
$$3s+3r_2-3\geq 3(r_1+1)+3r_2-3=k-1,$$ all elements of $\mathbb{Z}_k$ are covered.

If $(k-1)/3 + r_2-1 \leq s\leq (k-1)/3 + r_1-2$, then there is no gap between the first three intervals, so their union is $[0,2s+r_1+2r_2-3]$.  Here we have 
$$2s+r_1+2r_2-3 \geq 2(k-1)/3+2r_2-2+r_1+2r_2-3=k+3r_2-6\geq k-3.$$ 
If either of the inequalities is a strict inequality, then the union of these three intervals covers $\mathbb{Z}_k$ with the exception of at most one element.  On the other hand, if both inequalities are equalities, then we have $s=(k-1)/3$, $r_1=(k-4)/3$, and $r_2=1$; in this case we have $3(I_1 \cup I_2)=\mathbb{Z}_k \setminus \{k-2\}$.

If $(k-1)/3 + r_1-1 \leq s$, then either $s=(k-1)/3+r_1-1$ or $s=(k-1)/3+r_1$.  Note that if $r_1 \geq (k-1)/6+1$, then 
$s\leq (k-1)/3+r_1 \leq 3r_1-2,$ which means that there is no gap between the first two intervals, and thus they cover $[0,s+2r_1+r_2-3]$.  If we also have $s+r_1 \geq 2(k-1)/3+2$, then 
$$s+2r_1+r_2-3 \geq 2(k-1)/3+2+(k-1)/3-3=k-2,$$ and thus all elements of $\mathbb{Z}_k$ are covered with the possible exception of $k-1$.  If we still have $r_1 \geq (k-1)/6+1$ but $s+r_1 \leq 2(k-1)/3+1$, then we must have $r_1=(k-1)/6+1$ and $s=(k-1)/2$, so the first two intervals cover $[0,k-3]$, but the third interval includes $k-1$, and thus all elements of $\mathbb{Z}_k$ are covered with the possible exception of $k-2$.

This leaves us with only the cases when $r_1=r_2=(k-1)/6$, and $s=(k-1)/3+r_1-1=(k-3)/2$ or $s=(k-1)/2$. In the first case, we can compute that, as a set of integers, $3(I_1 \cup I_2)$ equals $$[0,2k-8] \setminus \{i(k-3)/2-1 \; : \; i=1,2,3 \}.$$  
For $k=7$, this means that $3(I_1 \cup I_2)=\{0,2,4,6\}$, so $|3(I_1 \cup I_2)| \neq k-2$.  When $k>7$, then $k+(k-3)/2-1$ is between $3(k-3)/2-1$ and $2k-8$, so we find that $3(I_1 \cup I_2)=\mathbb{Z}_k \setminus \{k-4\}$.  

The remaining case is when $r_1=r_2=(k-1)/6$ and $s=(k-1)/2$, in which case $I_1\cup I_2$ is an arithmetic progression with starting element $(k-1)/2$ and difference $(k+1)/2$, so $|3(I_1 \cup I_2)|=k-3$. 
\end{proof}

The only groups left to treat are the elementary abelian 7-groups, and they require considerable attention.  Our result will follow easily from the following structure theorem.

\begin{theorem}\label{thm-Z7r}
Let $r$ be a positive integer.  Suppose that $A$ is a subset of $G=\mathbb{Z}_7^r$ of size $(7^r-1)/3$ and $0 \not \in 3A$.  Then there is an ascending chain of subgroups $$\{0\} = H_0 < H_1 < \cdots < H_r = G$$ and elements
$$a_0,a_0'\in H_1,\quad a_k \in H_{k+1} \setminus H_k\ \; \mbox{for} \; k=1,\dots,r-1,$$
such that
$$A =\{a_0,a_0'\} \cup \bigcup\limits_{k=1}^{r-1}  \left(\{a_k, 2a_k\} + H_k \right).$$

\end{theorem}

\begin{proof}
First, recall that $\mathbb{Z}_7^r$ has exactly $(7^r-1)/6$ subgroups of index 7; indeed, identifying $\mathbb{Z}_7^r$ with the $r$-dimensional vector space over $\mathbb{Z}_7$, we note that each $(r-1)$-dimensional subspace corresponds to its normal vector that is unique up to nonzero scalar multiples.

Next, we prove that our conditions imply that for any subgroup $H$ of $G$ we have $|A\cap H|=(|H|-1)/3.$  Since by Proposition \ref{3-critical prop} we have $$\chi (H,3)=(|H|-1)/3+1,$$ we see that $H$ may contain at most $(|H|-1)/3$ elements of $A$, since otherwise $H \subseteq 3A$, contradicting $0
\not \in 3A$.  Therefore, we only need to prove that $H$ contains at least $(|H|-1)/3$ elements of $A$.  As $0 \not \in 3A$ implies that $0 \not \in A$, this trivially holds for $|H|=1$.  

For subgroups of order 7, we observe that the collection of {\em pierced lines} 
$$\{H \setminus \{0\}  \; : \; H \le G, |H|=7\}$$ forms a partition of $G \setminus \{0\}$.  Therefore, in order to have $|A|=(|G|-1)/3$, no pierced line, and thus no subgroup of order 7, may contain fewer than 2 elements of $A$.  Since for all subgroups $H$ of $G$, $H \setminus \{0\}$ is the disjoint union of pierced lines, our claim follows.

We are now ready to prove our theorem.  For $r=1$ there is nothing to prove.

We consider the case of $r=2$ next, and suppose that $A$ is a 16-element subset of $\mathbb{Z}_7^2$ such that $0 \not \in 3A$.  Note that if $H\leq \mathbb{Z}_7^2$ is of order 7, then at most two $H$-cosets can contain 3 or more elements from $A$. Suppose, to the contrary, that $H$-cosets $C_1, C_2,$ and $C_3$ each contain at least 3 elements from $A$.  Since $\chi (G/H,3)=\chi (\mathbb{Z}_7,3)=3$, we can then find (not necessarily distinct) indices $i,j,k \in \{1,2,3\}$ so that $C_i+C_j+C_k=H$.  Letting $A_i = A\cap C_i$, $A_j = A\cap C_j$, and $A_k = A\cap C_k$, Kneser's Theorem implies that 
$$|A_i + A_j + A_k| \geq |A_i|+|A_j|+|A_k| - 2|K|,$$ where $K$ is the stabilizer subgroup of $A_i+A_j+A_k$ in $H$.  Since $0 \not \in 3A$, here $A_i+A_j+A_k$ is a proper subset of $H$, and thus is aperiodic (that is, $K$ is trivial).  But then our inequality becomes $$ 6 \geq |A_i|+|A_j|+|A_k| - 2,$$ a contradiction. 

Next, we show that there is a subgroup $H$ of $G$ of order 7 so that one of its cosets contains at least 4 elements from $A$.  For the sake of contradiction, assume the contrary. Then for each $H$, out of the seven $H$-cosets, two contain 3 elements from $A$ and five contain 2 elements from $A$. Let us count the size of the following set in two different ways:
$$S:=\{(C, a, a')  \; : \; C \text{ is an affine line in }G;\ a, a' \in C \cap A; a\ne a'  \},$$ where by an {\em affine line} we mean a coset of a subgroup of order 7.  
On one hand, after arbitrarily choosing distinct elements $a$ and $a'$ from $A$, there exists a unique affine line $C$ through $a$ and $a'$, thus $|S|=|A|\cdot (|A|-1)=240$.

For a different count, we partition the 56 affine lines into 8 different {\em parallel classes} depending on which subgroup they correspond to.  According to our indirect assumption, for each such class, the numbers of elements of $A$ lying on the 7 parallel lines are $3, 3, 2, 2, 2, 2, 2$.  Therefore, for each class the number of suitable pairs $a, a'$ is $6+6+2+2+2+2+2=22$, yielding $|S|=8\cdot 22=176$, a contradiction.

Therefore, we may choose  a subgroup $H$ of order 7 in $G$ in such a way that at least one of its cosets contains at least 4 elements from $A$.  We choose an arbitrary element $c \in G \setminus H$, and let $C_i=ic+H$ for $i=0,\ldots,6$; we also set $A_i = C_i \cap A$.   
According to our considerations at the beginning of the proof, we have $|A_0|=2$, and we may assume that $|A_1| = \max \{ |A_i|\}$; by the previous reasoning, $|A_1| \geq 4$.

An argument similar to the one above using Kneser's Theorem yields that when there are (not necessarily distinct) indices $i, j, k \in \{0, \dots, 6\}$ for which $i+j+k \equiv 0$ mod 7, none of $A_i$, $A_j$, or $A_k$ is the emptyset, and $|A_i|+|A_j|+|A_k| \geq 9$, then $H \subseteq 3A$, contradicting $0 \not \in 3A$.  Therefore, we have $2|A_1|+|A_5| \leq 8$ if $A_5 \neq \emptyset$; since $|A_1| \geq 4$, this yields $A_5=\emptyset$.  Similarly, $|A_0|+|A_1|+|A_6| \leq 8$ if $A_6 \neq \emptyset$, and thus $|A_6| \leq \max \{0,6-|A_1|\}$; and $|A_1|+2|A_3| \leq 8$, and thus $|A_3| \leq 4 - |A_1|/2$.  Furthermore, we can easily see that $|A_2|+|A_4| \leq |A_1|$; indeed, if neither $A_2$ nor $A_4$ is empty, then this follows from $|A_2|+|A_4| \leq 8 - |A_1|$ since $|A_1| \geq 4$, and it holds trivially when one of $A_2$ or $A_4$ is empty, by our choice of $A_1$.  We thus have 
\begin{eqnarray*} 16=|A| & = & |A_0|+|A_1| + |A_3|+|A_5|+|A_6|  +(|A_2|+|A_4|) \\
& \leq & 2 + |A_1| + (4 - |A_1|/2) + 0 + \max \{0,6-|A_1|\} + |A_1|,
\end{eqnarray*} from which we get $|A_1|=7$.  But then our previous inequalities yield $A_3=A_6=\emptyset$ and $|A_2|+|A_4|=7$; the latter equality can only occur when one of $A_2$ or $A_4$ is empty and the other is a full coset.  

Note that $(C_0, C_1, C_2)$ and $(C_0,C_4,C_1)$ are both 3-term arithmetic progressions in $G/H$.  Let us now set $H_1=H$, $\{a_0, a_0'\}=A_0$, and $a_1=c$ or $a_1=4c$ depending on whether $A=A_0 \cup C_1 \cup C_2$ or $A=A_0 \cup C_1 \cup C_4$.  Then 
$$A=\{a_0,a_0'\} \cup \left(\{a_1, 2a_1\} + H_1 \right),$$ 
and thus our proof for the case of $r=2$ is complete.

We now use induction to prove that our result holds for any $r \geq 3$.  
To start, we examine cosets of subgroups of rank $r-2$ in $G$, which here we call flats; more specifically, we say that a coset of a subgroup $K$ of rank $r-2$ is a {\em flat of type $K$}.  We can count the number of flats fully contained in $A$ as follows.  Since none of them is a subgroup, each flat $F$ contained in $A$ generates a unique subgroup $\langle F \rangle$ of index 7.  There are $(7^r-1)/6$ subgroups of index 7 in $G$,  and by our inductive hypothesis, each such subgroup contains exactly two flats that are in $A$.  Therefore, $A$ contains exactly $(7^r-1)/3$ flats; we call these {\em $A$-flats}.      

We see that not all $A$-flats are of the same type: indeed, a subgroup of rank $r-2$ in $G$ has $49$ cosets, of which at most $48$ are in $A$, but $(7^r-1)/3$ is more than $48$ if $r \geq 3$.  Now let $F_1$ and $F_2$ be $A$-flats of types $K_1$ and $K_2$, respectively, with $K_1 \neq K_2$.  Then $H=K_1+K_2$ is a subgroup of index 7 in $G$, since $K_1+K_2=G$ would imply that $F_1+2F_2=G$, contradicting $3A \neq G$.  For the same reason, $H$ contains every subgroup of rank $r-2$ that has a flat in $A$.    

Now let $c \in G \setminus H$ be an arbitrary element; the cosets of $H$ in $G$ then are $C_i=ic+H$ as $i=0,1,\ldots, 6$.  Note that each $A$-flat is contained entirely in one of the seven cosets of $H$ in $G$; let ${\mathcal F}_i$ be the union of $A$-flats in $C_i$.     
By our inductive assumption, $H$ itself has exactly two $A$-flats, and they are of the same type.   However, there has to be at least one coset of $H$ that has at least two $A$-flats of different types, since we have more than $2+6 \cdot 7=44$ $A$-flats;   
without loss of generality, suppose that $C_1$ contains at least two different types of $A$-flats.  

Note that the sum of two flats of different types is an entire coset of $H$.  Therefore, ${\mathcal F}_6 = \emptyset$, since otherwise ${\mathcal F}_0 + {\mathcal F}_1 + {\mathcal F}_6=C_0$, contradicting $0 \not \in 3A$.  Similarly, from $1+3+3 \equiv 1+1+5 \equiv 1+2+4 \equiv  0$ mod 7, we get ${\mathcal F}_3 = {\mathcal F}_5 = \emptyset$ and that at least one of ${\mathcal F}_2$ or ${\mathcal F}_4$ is empty.  So either $C_0 \cup C_1 \cup C_2$ or $C_0 \cup C_1 \cup C_4$ contains all $A$-flats; since $C_0$ contains exactly 2, the other two cosets each have to contain the maximum possible number that they can, which is $7 \cdot (7^{r-1}-1)/6$.  But if a coset of $H$ contains 7 $A$-flats of the same type, then it is the disjoint union of these $A$-flats, so we must have $A=(A \cap H) \cup (c+H) \cup (2c+H)$ or $A=(A \cap H) \cup (c+H) \cup (4c+H)$.  This means that we can set $H_r=H$ and $a_r=c$ or $a_r=4c$, and then apply the inductive hypothesis within $H$.  This completes our proof.
\end{proof}

\begin{corollary}\label{cor-z7n} If $G$ is an elementary abelian 7-group, then $S(G,3)=\{n-3\}$.

\end{corollary}

\begin{proof}  Let $A$ be a 3-incomplete subset of $G$ of size $(n-1)/3$.  After translating $A$, if needed, we may assume that $0\not\in 3A$; we can then use  Theorem~\ref{thm-Z7r} to show that $|3A|=n-3$.  Indeed, we find that if $n=7^r$, then 
$$|3A| = 6 \cdot 7^{r-1} + 6 \cdot 7^{r-2} + \cdots + 6 \cdot 7 + 6 - 2 =7^r-3.$$ 
\end{proof}

\section*{Acknowledgments}
P.~P.~P. was supported by the Lend\"ulet program of the Hungarian Academy of Sciences (MTA) and by the National Research, Development and Innovation Office NKFIH (Grant Nr. K124171 and K129335).

\end{document}